\documentclass[11pt]{article}
\usepackage{mathrsfs}
\usepackage{amsfonts}
\usepackage{amssymb,amsmath,amsthm}
\newcommand{\R}{\mathbb{R}}

\newcommand{\N}{\mathbb{N}}
\def\qed{\hfill $\Box$ \smallskip}
\newtheorem{theorem}{Theorem}[section]

\newtheorem{lemma}[theorem]{Lemma}
\newtheorem{proposition}{Proposition}

\theoremstyle{definition}
\newtheorem{definition}[theorem]{Definition}
\newtheorem{remark}{Remark}

\begin{document}
\title{\LARGE\bf{Normalized ground states solutions for nonautonomous Choquard equations}$\thanks{{\small This work was partially supported by NSFC(11901532,11901531).}}$ }
\date{}
 \author{ Huxiao Luo, Lushun Wang$\thanks{{\small Corresponding author. E-mail: luohuxiao@zjnu.edu.cn (H. Luo), lushun@zjnu.edu.cn (L. Wang).}}$\\
\small Department of Mathematics, Zhejiang Normal University, Jinhua, Zhejiang, 321004, P. R. China
}
\maketitle
\begin{center}
\begin{minipage}{13cm}
\par
\small  {\bf Abstract:} In this paper, we study normalized ground state solutions for the following nonautonomous Choquard equation:
\[
\left\{
\begin{array}{ll}
\aligned
&-\Delta u-\lambda u=\left(\frac{1}{|x|^{\mu}}\ast A|u|^{p}\right)A|u|^{p-2}u,\\
&\int_{\mathbb{R}^{N}}|u|^{2}dx=c,\quad u\in H^1(\mathbb{R}^N,\R),
\endaligned
\end{array}
\right.
\]
where $c>0$, $0<\mu<N$, $\lambda\in\R$, $A\in C^1(\R^N,\R)$. For $p\in(2_{*,\mu}, \bar{p})$, we prove that the Choquard equation possesses ground state normalized solutions, and the set of ground states is orbitally stable.
 For $p\in (\bar{p},2^*_\mu)$, we find a normalized solution, which is not a global minimizer. $2^*_\mu$ and $2_{*,\mu}$ are the upper and lower critical exponents due to the Hardy-Littlewood-Sobolev inequality, respectively. $\bar{p}$ is $L^2-$critical exponent.
Our results generalize and extend some related results.
 \vskip2mm
 \par
 {\bf Keywords:} Nonautonomous Choquard equation; Variational methods; Normalized solution; Orbitally stable.

 \vskip2mm
 \par
 {\bf MSC(2010): }35J50; 58E30

\end{minipage}
\end{center}

 {\section{Introduction}}
 \setcounter{equation}{0}
Consider the time dependent nonautonomous Choquard equation
\begin{equation}\label{t1.1.0}
\left\{
\begin{array}{ll}
\aligned
&i\partial_t\psi=-\Delta \psi-\left(\frac{1}{|x|^\mu}\ast A|\psi|^{p}\right)A|\psi|^{p-2}\psi,\quad t\in\R,~ x\in\R^N,\\
&\psi(0,x)=\psi_0(x)\in H^1(\R^N,\mathbb{C}),
\endaligned
\end{array}
\right.
\end{equation}
where $N\in\mathbb{N}$ denotes space dimension, $0<\mu<N$, $A\in L^\infty(\R^N,\R)$, $p\in(2_{*,\mu},2^*_\mu)$, where
\begin{equation*}
\left\{
\begin{array}{ll}
\aligned
&2_{*,\mu}:=\frac{2N-\mu}{N}, \\
& 2^*_\mu:=\frac{2N-\mu}{(N-2)_+}=
\left\{
\begin{array}{ll}
\aligned
&\frac{2N-\mu}{N-2}~\text{if}~N\geq 3,\\
&+\infty~\text{if}~N=1,2.
\endaligned
\end{array}
\right.
\endaligned
\end{array}
\right.
\end{equation*}

Equation $(\ref{t1.1.0})$ has several physical origins. In particular, when $N = 3$, $p = 2$, $\mu = 1$ and $A(x)\equiv 1$, (\ref{t1.1.0}) appeared at least as early as in 1954, in a work by S. I. Pekar \cite{MR2561169,pekar} describing the quantum mechanics of a polaron at rest. In 1976, P. Choquard \cite{MR471785} used $(\ref{t1.1.0})$ to describe an electron trapped in its own hole, in a certain approximation to Hartree-Fock theory of one component plasma.
Twenty years later, R. Penrose proposed $(\ref{t1.1.0})$ as a model of self-gravitating matter, in a programme in
which quantum state reduction is understood as a gravitational phenomenon, see \cite{MR1649671}.

For our setting, $A(x)$ is a real-valued bounded function and not necessarily a constant function. However, according to \cite{cazenave,feng}, by testing equation (\ref{t1.1.0}) against $\bar{\psi}$ (the complex conjugate of $\psi$) and $\partial_t\bar{\psi}$, it is easy to obtain the conservation property of mass $\int_{\R^N}|\psi|^2dx$ and of energy
$$\frac{1}{2}\int_{\R^N}|\nabla \psi|^2dx-\frac{1}{2p}\int_{\R^N}(|x|^{-\mu}\ast A|\psi|^p)A|\psi|^pdx.$$
And similar to \cite[Theorem 1.1]{feng}, for $0<\mu<\min\{N,4\}$ and $2\leq p<2^*_\mu$, we have from Hardy-Littlewood-Sobolev inequality and H\"{o}lder inequality that
$$\|(|x|^{-\mu}\ast A|u|^p)A|u|^{p-2}u -(|x|^{-\mu}\ast A|v|^p)A|v|^{p-2}v\|_{\frac{2Np}{2Np-2N+\mu}}
\leq C(\| u\|^{2p-2}_{\frac{2Np}{2N-\mu}} + \|v\|^{2p-2}_{\frac{2Np}{2N-\mu}} )\|u-v\|_{\frac{2Np}{2N-\mu}},$$
where $\|\cdot\|_{q}$ denotes the usual norm of the Lebesgue space $L^{q}(\mathbb{R}^N,\R)$.
Then by Strichartz estimate and fixed point argument \cite[Theorem 3.3.9]{cazenave}, (\ref{t1.1.0}) is local well-posedness in $H^1(\R^N,\mathbb{C})$. Moreover, if $2\leq p<\bar{p}$, it is standard to get global existence by the conservation of mass and energy and Gagliardo-Nirenberg inequality of convolutional type, see \cite[Theorem 1.2]{feng}.
Here, $\bar{p}:=\frac{2N-\mu+2}{N}$ denotes the $L^2$-critical (mass-critical) exponent.

In general, equation (\ref{t1.1.0}) admits special regular solutions, which are called solitary (standing)
waves. More precisely, these solutions have the form $\psi(t, x) = e^{-i\lambda t}u(x)$, where $-\lambda\in\R$
is the frequency and $u(x)$ solves the following elliptic equation
\begin{equation}\label{1.1.0}
\left\{
\begin{array}{ll}
\aligned
&-\Delta u-\lambda u=\left(\frac{1}{|x|^{\mu}}\ast A|u|^{p}\right)A|u|^{p-2}u \quad \text{in}~\R^N, \\
&\int_{\mathbb{R}^{N}}|u|^{2}dx=\int_{\R^N}|\psi_0|^2dx:=c,\\
&u\in H^1(\mathbb{R}^N,\R).
\endaligned
\end{array}
\right.   \tag{P}
\end{equation}
Here the constraint $\int_{\mathbb{R}^{N}}|u|^{2}dx=c$ is natural due to the conservation of mass.

To study (\ref{1.1.0}) variationally, we need to recall the following Hardy--Littlewood--Sobolev inequality \cite[Theorem 4.3]{MR1817225}.
\begin{lemma}
 (Hardy--Littlewood--Sobolev inequality.)  Let $s,r>1$ and $0<\mu<N$ with $1/s+\mu/N+1/r=2$, $f\in L^{s}(\R^N,\R)$ and $h\in L^{r}(\R^N,\R)$. There exists a sharp constant $C(N,\mu,s,r)$, independent of $f,h$, such that
\begin{equation}\label{HLS1}
\int_{\mathbb{R}^{N}}\int_{\mathbb{R}^{N}}\frac{f(x)h(y)}{|x-y|^{\mu}}dxdy\leq C(N,\mu,s,r) \|f\|_{s}\|h\|_{r}.
\end{equation}
\end{lemma}
If $A(x)$ is bounded in $\R^N$, then by (\ref{HLS1}) and Sobolev inequality, the integral
$$
\int_{\mathbb{R}^{N}}\int_{\mathbb{R}^{N}}\frac{ A(x) |u(x)|^pA(y) |u(y)|^p}{|x-y|^\mu} \mathrm{d} x\mathrm{d} y
$$
is well defined in $H^1(\R^N,\R)$ for
$$
2_{*,\mu}=\frac{2N-\mu}{N}\leq p\leq2^*_\mu=\frac{2N-\mu}{(N-2)_+}.
$$
As a result, the functional $I: H^1(\R^N,\R)\mapsto\R$,
\begin{equation}\label{I}
I(u):=\frac{1}{2} \int_{\mathbb{R}^{N}}|\nabla u|^{2} \mathrm{d} x-\frac{1}{2p}\int_{\mathbb{R}^{N}}\int_{\mathbb{R}^{N}}\frac{ A(x) |u(x)|^pA(y) |u(y)|^p}{|x-y|^\mu} \mathrm{d} x\mathrm{d} y
\end{equation}
is well defined. Furthermore, by a standard argument, we have $I \in C^1(H^1(\R^N,\R), \R)$.

Due to the constraint $\int_{\mathbb{R}^{N}}|u|^{2}dx=c$, the solution for \eqref{1.1.0} is called normalized solution, which can be found by looking for critical points of the functional $I$ on the constraint
$$
\mathcal{S}(c)=\left\{u \in H^{1}\left(\mathbb{R}^{N},\R\right) :\int_{\R^N}|u|^2dx=c\right\}.
$$
In this situation, the frequency $-\lambda\in\R$ can no longer be fixed but instead appears
as a Lagrange multiplier, and each critical point $u_c \in \mathcal{S}(c)$ of $I|_{\mathcal{S}(c)}$ corresponds a
Lagrange multiplier $\lambda_c \in \R$ such that $(u_c, \lambda_c)$ solves (weakly) \eqref{1.1.0}.
Due to physical application, we are particularly interested in normalized ground state solutions, defined as follows:
\begin{definition}
For any fixed $c > 0$, we say that $u_c\in \mathcal{S}(c)$ is a normalized
ground state solution to (\ref{1.1.0}) if $I'|_{\mathcal{S}(c)}(u_c) = 0$ and
$$I(u_c) = \inf\{I(u):  u \in \mathcal{S}(c), I'|_{\mathcal{S}(c)}(u) = 0\}.$$
\end{definition}
For any $c>0$, we set $$\sigma(c):=\inf\limits_{u\in \mathcal{S}(c)} I(u).$$
If the minimizers of $\sigma(c)$ exist, then all minimizers are critical points
of $I|_{\mathcal{S}(c)}$ as well as normalized ground state solutions to \eqref{1.1.0}.
\begin{remark}
If $\sigma(c)$ admits a global minimizer, then this definition of ground
states naturally extends the notion of ground states from linear quantum mechanics.
\end{remark}
There is a lot of literature studying ground states to the autonomous Choquard equations. For example, the existence
of ground states to the autonomous Choquard equation
\begin{equation}\label{e1.1.2}
-\Delta u-\lambda u=\left(\frac{1}{|x|^{\mu}}\ast |u|^{p}\right)|u|^{p-2}u \quad \text{in}~\R^N
\end{equation}
is established by Moroz and Van Schaftingen \cite{MR3056699} under $\lambda=-1$ and $2_{*,\mu}< p<2^*_\mu$. In \cite{MR3642765}, Ye obtained sharp existence results of the normalized solution to (\ref{e1.1.2}). Precisely,
\begin{itemize}
\item[(i)] If $p \in \left(2_{*,\mu}, \bar{p}\right)$, $\sigma(c)$ has at least one
minimizer for each $c > 0$ and $\sigma(c) > -\infty$;
\item[(ii)] If $p\in\left(\bar{p}, 2^*_\mu\right)$, $\sigma(c)$ has no minimizer for each $c > 0$ and $\sigma(c)=-\infty$;
\item[(iii)] The $L^2-$critical case $p = \bar{p}$ is complicated, see \cite{MR3642765} for details.
\end{itemize}

As far as we know, normalized solution of nonautonomous Choquard equation \eqref{1.1.0} has not been studied.
In this paper, we are interested in normalized solutions for the nonautonomous Choquard equation (\ref{1.1.0}) under two cases: (i) $L^2-$subcritical case, i.e., $p\in(2_{*,\mu},\bar{p})$; (ii) $L^2-$supercritical case, i.e., $p\in(\bar{p}, 2^*_\mu)$.

For the $L^2-$subcritical case, we generalize the result in \cite{MR3642765} to the nonautonomous setting.
\begin{theorem}\label{th1.1} Let $N\geq1$, $0<\mu<N$ and $2_{*,\mu}<p<\bar{p}$. Suppose that
\begin{itemize}
\item[($A_{1}$)] $A\in C^1(\R^N, \R)$, $\lim\limits_{|x|\to+\infty}A(x):=A_\infty\in(0,+\infty)$, and
$A(x)\geq A_\infty$ for all $x\in\R^N$;
\item[($A_{2}$)] there exists a constant $\varrho> 0$ such that $t^{\frac{N-\mu+2\varrho(p-1)}{2}}A(t x)$ is nondecreasing on $t \in (0, +\infty)$ for every $x\in \R^N$.
\end{itemize}
Then $I$ admits a critical point $\bar{u}_c$ on $\mathcal{S}(c)$ which is a negative global minimum of $I$. Moreover, for the above critical point $\bar{u}_c$, there exists Lagrange multiplier $\lambda_c $ such that $(\bar{u}_c, \lambda_c)$ is
a solution of $(\ref{1.1.0})$.
\end{theorem}

\begin{remark}
For autonomous situation $A\equiv1$, Ye \cite{MR3642765} proved that the Lagrange multiplier $\lambda_c<0$. However, in our nonautonomous setting, we cannot be sure that Lagrange multiplier $\lambda_c $ is negative due to the complexity in the Poho\v{z}aev identity of nonautonomous equation.
\end{remark}

Compared with \cite{MR3642765}, the proof of Theorem \ref{th1.1} is more complex due to more general
nonlinearity in (\ref{1.1.0}). The main difficulty is to prove the compactness of a minimizing
sequence of $\sigma(c)= \inf\limits_{u\in \mathcal{S}(c)}I$. To do that, inspired by \cite{MR2826402,MR4081327,MR778970,MR778974}, we shall establish the following subadditivity inequality:
\begin{equation}\label{c1.11}
\sigma(c)<\sigma(\alpha)+\sigma(c-\alpha), \quad \forall 0<\alpha<c
\end{equation}
with the help of the scaling
\begin{equation}\label{A1}
s\mapsto u_{s}:=s^{\varrho}u(x/s).
\end{equation}

 Let $Z_c$ denote the set of the normalized ground state solutions for (\ref{t1.1.0}). We also interest in the stability and instability of normalized ground state solutions, defined as follows:
\begin{definition} $Z_{c}$ is orbitally stable if for every $\varepsilon > 0$ there exists $\delta > 0$ such that, for
any $\psi_0 \in H^1(\R^N, \mathbb{C})$ with $\inf\limits_{u\in Z_{c}} \|\psi_0 - u\|_{H^1(\R^N, \mathbb{C})} < \delta$, we have
$$\inf\limits_{u\in Z_{c}}
\|\psi(t,\cdot)-u\|_{H^1(\R^N, \mathbb{C})}< \varepsilon \quad \forall t > 0,$$
where $\psi(t, \cdot)$ denotes the solution to (\ref{t1.1.0}) with initial datum $\psi_0$.
A standing wave $e^{-i\lambda t} u$ is strongly unstable if for every $\varepsilon > 0$ there exists $\psi_0\in H^1(\R^N, \mathbb{C})$
such that $\|u -\psi_0\|_{H^1(\R^N, \mathbb{C})} < \varepsilon$, and $\psi(t, \cdot)$ blows-up in finite time.
\end{definition}

Following the same argument as in \cite{MR677997}, we can deduce that $Z_c$ is orbitally stable provided that any minimizing sequence to $\sigma(c)$ is compact in $H^1(\R^N,\R)$.

Note that due to the presence of the coefficients $A(x)$ in (\ref{1.1.0}), our minimization problems are
 not invariant by the action of the translations. To overcome this difficult, we adopt the method of studying the nonautonomous Schr\"{o}dinger equation in \cite{bellazzini}. The main point is the analysis of the compactness of minimizing sequences
 to suitable constrained minimization problem related to (\ref{1.1.0}).

More precisely,
\begin{theorem}\label{th1.1.5}
Let $N\geq1$, $0<\mu<2$, $2\leq p<\bar{p}$. Suppose that
\begin{itemize}
\item[($A'_{1}$)] $A\in L^\infty(\R^N, \R)$, $A(x)\geq0$ for almost every $x\in\R^N$, and there is $A_0>0$ such that
$meas\{ A(x)>A_0\}\in(0,+\infty)$.
\end{itemize}
Then there exists $c_0> 0$ such that all the minimizing sequences for $I|_{\mathcal{S}(c)}$ are compact
 in $H^1(\R^N,\R)$ provided that $c> c_0$. In particular, $Z_c$ is a nonempty compact set
 and it is orbitally stable.
\end{theorem}
\begin{remark}
The condition $0<\mu<2$ in Theorem \ref{th1.1.5} is to ensure $2<\bar{p}=\frac{2N-\mu+2}{N}$.
If $p<2$, the nonlocal term $\left(|x|^{-\mu}\ast A|\psi|^p\right)A|\psi|^{p-2}\psi$ in dispersion equation (\ref{t1.1.0}) is singular, the existence of local-well posedness of (\ref{t1.1.0}) is invalid, similar to the autonomous equation ($A\equiv 1$) in \cite{feng}.
\end{remark}

In the second part of this article, we consider the $L^2-$supercritical case. Since $\sigma(c) = -\infty$ for $p\in(\bar{p}, 2^*_\mu)$, it is impossible to search for a minimum of $I$ on $\mathcal{S}(c)$. So it is nature to look for a critical point of $I$ having a minimax characterization. For example, for the following Schr\"{o}dinger equation
\begin{equation}\nonumber
-\Delta u-\lambda u=f(u),~~
u\in H^{1}(\R^N,\R),
\end{equation}
Jeanjean \cite{MR1430506} constructed mountain-pass geometrical structure on
$\mathcal{S}(c)\times \R$ to an auxiliary functional
\begin{equation}\label{c1.8}
\tilde{I}(u, t)=\frac{e^{2 t}}{2} \int_{\mathbb{R}^{N}}|\nabla u|^{2} \mathrm{d} x-\frac{1}{e^{t N}} \int_{\mathbb{R}^{N}} F\left(e^{\frac{t N}{2}} u\right) \mathrm{d} x,
\end{equation}
where $F(u)=\int_0^uf(t)dt$.
 Then applying the Ekeland principle to the auxiliary functional, the author obtained a
sequence $\{(v_n,s_n)\} \subset \mathcal{S}(c)\times\R$ which can be used to construct a bounded Palais-Smale sequence $\{u_n\}\subset \mathcal{S}(c)$ for $I$ at the M-P level.

By using Jeanjean's method \cite{MR1430506}, Li and Ye \cite{MR3390522} obtained the normalized solutions
to the Choquard equation:
\begin{equation}
-\Delta u-\lambda u=\left(|x|^{-\mu}\ast F(u)\right)f(u),
\end{equation}
where $\lambda\in\R$, $N \geq3$, $\mu\in(0,N)$, and $F(u)$ behaves like $|u|^p$ for $\frac{2N-\mu+2}{N}<p<\frac{2N-\mu}{N-2}$.

However, for nonautonomous equation (\ref{1.1.0}), the method of constructing a Poho\u{z}aev-Palais-Smale sequence in \cite{MR1430506} fails. To overcome this difficulty, we adopt the method in \cite{MR4081327}. More precisely, we assume
that
\begin{itemize}
\item[($A_{3}$)] $t \mapsto (Np-2N+\mu)A(t x) - 2\nabla A(t x) \cdot (t x)$
is nonincreasing on $(0,\infty)$ for every $x\in\R^N$;
\item[($A_{4}$)] $t^{\frac{2p-(Np-2N+\mu)}{2}}A(t x)$ is strictly increasing on $t \in (0, \infty)$ for every $x \in \R^N$.
\end{itemize}
Besides $A\equiv$constant, there are indeed many functions which satisfy
$(A_1), (A_3)$ and $(A_4)$. For example
\begin{itemize}
\item[($i$)] $A_1(x) = 1 + be^{-\tau|x|}$ with $0 < b \leq e\cdot\frac{2p- (Np- 2N+\mu)}{2}$ and $\tau> 0$;
\item[($ii$)] $A_2(x) = 1 + \frac{b}{1+|x|}$ with $0 < b \leq 2[2p- (Np- 2N+\mu)]$.
\end{itemize}
Under $(A_1), (A_3)$ and $(A_4)$, we shall establish the existence of normalized ground
state solutions to the nonautonomous Choquard equation (\ref{1.1.0}) by taking a minimum
on the manifold
\begin{equation}\label{c1.9}
\mathcal{M}(c)=\left\{u \in \mathcal{S}(c) : J(u) :=\left.\frac{\mathrm{d}}{\mathrm{d} t} I\left( u^{t}\right)\right|_{t=1}=0\right\},
\end{equation}
where $u^{t}(x) :=t^{N / 2}u(t x)$ for all $t > 0$ and $x\in\R^N$, and $u^t \in \mathcal{S}(c)$ if $u \in \mathcal{S}(c)$.
\begin{theorem}\label{th1.2} Suppose that $N\geq1$, $0<\mu<N$, $\bar{p}<p<2^*_\mu$, $(A_1), (A_3)$ and $(A_4)$ hold. Then for any $c > 0$, (\ref{1.1.0})
has a couple of solutions $\left(\overline{u}_{c}, \lambda_{c}\right) \in \mathcal{S}(c) \times \mathbb{R}$ such that
$$
I\left(\overline{u}_{c}\right)=\inf _{u \in \mathcal{M}(c)} I(u)=\inf _{u \in \mathcal{S}(c)} \max _{t>0} I\left( u^{t}\right)>0.
$$
\end{theorem}

To address the lack of compactness, we should consider the \emph{limit equation} of (\ref{1.1.0}):
\begin{equation}\label{1.1.2}
-\Delta u-\lambda u=A_\infty^2\left(|x|^{-\mu}\ast|u|^p\right)|u|^{p-2}u,~~
u\in H^{1}(\R^N).
 \tag{P0}
\end{equation}
The energy functional is defined as follows:
\begin{equation}\label{c1.13}
I_\infty(u)=\frac{1}{2} \int_{\mathbb{R}^{N}}|\nabla u|^{2} \mathrm{d} x-\frac{A_\infty^2}{2p}\int_{\mathbb{R}^{N}}\int_{\mathbb{R}^{N}}\frac{ |u(x)|^p |u(y)|^p}{|x-y|^\mu} \mathrm{d} x\mathrm{d} y.
\end{equation}
Similar to (\ref{c1.9}), we define
\begin{equation}\label{c1.14}
\mathcal{M}_\infty(c)=\left\{u \in \mathcal{S}(c) : J_\infty(u) :=\left.\frac{\mathrm{d}}{\mathrm{d} t} I_\infty\left( u^{t}\right)\right|_{t=1}=0\right\}.
\end{equation}
\begin{remark}
Compared to \cite{MR4081327}, the main difficulty in our nonlocal setting:
When proving that $\inf _{u \in \mathcal{M}(c)} I(u)$ can be achieved, it needs to be compared (\ref{1.1.0}) with the limit equation. The difference between $I$ and $I_\infty$ is more complicated than that of the Schr\"{o}dinger equation.
\end{remark}

Finally, we give our future research directions about this article: \\
For $A\equiv1$ and $\bar{p}<p<2^*_\mu$, by using the blow up for a class of initial data with nonnegative energy, Chen and Guo \cite{chen} proved that the standing wave of (\ref{t1.1.0}) must be strongly unstable. For nonautonomous situation ($A\not\equiv$constant), the method in \cite{chen} is invalid. We will study the problem in the future.

This paper is organized as follows. In section 2, we prove Theorem \ref{th1.1} and Theorem \ref{th1.1.5}. In section 3, we show Theorem \ref{th1.2}.

\vskip2mm
 \par\noindent
 In this paper, we make use of the following notation:  \\
$\diamondsuit$ $C, C_i , i = 1, 2, \cdot\cdot\cdot,$ will be repeatedly used to denote various positive constants whose exact values are irrelevant. \\
$\diamondsuit$ \begin{equation*}
2^*=
\left\{
\begin{array}{ll}
\aligned
&\frac{2N}{N-2}~&\text{if}~N\geq 3\\
&+\infty~&\text{if}~N=1,2
\endaligned
\end{array}
\right.
\end{equation*} denotes the Sobolev critical exponent. \\
$\diamondsuit$ $o(1)$ denotes the infinitesimal as $n\to+\infty$.\\
$\diamondsuit$ For the sake of simplicity, integrals over the whole $\R^N$ will be often written $\int$.

\vskip4mm
{\section{ $L^2-$subcritical case }}
 \setcounter{equation}{0}
First, we prove a nonlocal version of Brezis-Lieb lemma, which will be used in the proof below both $L^2-$subcritical case and $L^2-$supercritical case. We need the following classical Brezis-Lieb lemma \cite{bogachev}.
\begin{lemma} (\cite{bogachev}) Let $N\in\mathbb{N}$ and $q\in [2, 2^*]$.
If $u_n\rightharpoonup u$ in $H^1(\R^N,\R)$, then
\begin{equation}\label{BL}
\aligned
\int|u_{n}-u|^qdx-\int|u_{n}|^qdx=\int|u|^qdx+o(1).
\endaligned
\end{equation}
\end{lemma}
\begin{lemma}\label{lm2.11} Let $N\in\mathbb{N}$, $\mu\in(0, N)$, $p\in [2_{*,\mu}, 2^*_\mu]$, and $A,B\in L^\infty(\R^N,\R)$.
If $u_n\rightharpoonup u$ in $H^1(\R^N,\R)$, then
\begin{equation}\label{c2.18}
\aligned
&\int(|x|^{-\mu}\ast A|u_n-u|^p)B|u_{n}-u|^pdx-\int(|x|^{-\mu}\ast A|u_n|^p)B|u_{n}|^pdx\\
=&\int(|x|^{-\mu}\ast A|u|^p)B|u|^pdx+o(1).
\endaligned
\end{equation}
\end{lemma}
\begin{proof} For every $n\in\mathbb{N}$, one has
\begin{equation*}
\aligned
&\int(|x|^{-\mu}\ast A|u_n|^p)B|u_{n}|^pdx-\int_{\R^N}(|x|^{-\mu}\ast A|u_n-u|^p)B|u_{n}-u|^pdx\\
=&\int(|x|^{-\mu}\ast A(|u_n|^p-|u_n-u|^p))B(|u_n|^p-|u_n-u|^p)dx \\
&+2\int(|x|^{-\mu}\ast A(|u_n|^p-|u_n-u|^p))B|u_n-u|^pdx\\
&-\int(|x|^{-\mu}\ast A|u_n|^p)B|u_n-u|^pdx\\
&+\int(|x|^{-\mu}\ast A|u_n-u|^p)B|u_n|^pdx\\
:=&I_1+I_2+I_3+I_4.
\endaligned
\end{equation*}
By the classical Brezis-Lieb lemma with $q = \frac{2Np}{2N-\mu}$, we have $|u_n - u|^p - |u_n|^p \to |u|^p$, strongly
in $L^{\frac{2N}{2N-\mu}}(\R^N)$ as $n \to\infty$.
Then, Hardy-Littlewood-Sobolev inequality implies
that $$|x|^{-\mu}\ast A(|u_n - u|^p - |u_n|^p)\to |x|^{-\mu}\ast A|u|^p~~\text{in}~L^{\frac{2N}{\mu}}(\R^N,\R)~\text{as}~n \to\infty.$$ Thus
$$I_1\to \int(|x|^{-\mu}\ast A|u|^p)B|u|^pdx~\text{as}~n \to\infty.$$
On the other hand, $I_2$, $I_3$ and $I_4$ both converge to $0$ since that $|u_n - u|^p\rightharpoonup0$ weakly in $L^{\frac{2N}{2N-\mu}}(\R^N,\R)$, $|x|^{-\mu}\ast A|u_n|^p$ and $B|u_n|^p$ are bounded in $L^{\frac{2N}{\mu}}(\R^N,\R)$.
Thus, (\ref{c2.18}) holds.
\end{proof}
 \vskip4mm
{\subsection{The proof of Theorem \ref{th1.1} }}
In this section, we prove Theorem \ref{th1.1} under the conditions $(A_1)$-$(A_2)$ and $p\in(2_{*,\mu}, \bar{p})$. Since $A(x) \equiv A_\infty$ satisfies $(A_1)$-$(A_2)$, all the following conclusions on $I$ are also true for $I_\infty$.

For $u \in \mathcal{S}(c)$, set $u^s(x) = s^{\frac{N}{2}} u(sx)$ $\forall s> 0$. Then
$$\|u^s\|_2^2=\|u\|_2^2=c, \quad \|\nabla u^s\|_2^2=s^2\|\nabla u\|_2^2,$$
\begin{equation}\label{I2}
I(u^s)=\frac{1}{2}s^2\|\nabla u\|_2^2-\frac{1}{2p}s^{Np-2N+\mu}\int\int\frac{ A(s^{-1}x)|u(x)|^p A(s^{-1}y) |u(y)|^p}{|x-y|^\mu}dxdy.
\end{equation}
and
\begin{equation}\label{J}
\aligned
J(u)=&\frac{dI(u^s)}{ds}|_{s=1}\\
=&\|\nabla u\|_2^2-\frac{1}{2p}\int\int\frac{ \left[(Np-2N+\mu)A(x)-2\nabla A(x)\cdot x\right]A(y)|u(x)|^p  |u(y)|^p}{|x-y|^\mu}dxdy.
\endaligned
\end{equation}

\begin{lemma}\label{lm3.1} For any $c > 0$, $\sigma(c) =
\inf\limits_{u\in \mathcal{S}(c)}I(u)$ is well defined and $\sigma(c) < 0$.
\end{lemma}
\begin{proof}
By the Gagliardo-Nirenberg inequality
\begin{equation}\label{GN2}
\|u\|_q\leq C(N,q)\|\nabla u\|_2^{\frac{N(q-2)}{2q}}\|u\|_2^{1-\frac{N(q-2)}{2q}}~\forall q\in(2,2^*),
\end{equation}
Hardy--Littlewood--Sobolev inequality and $(A_1)$, for $u\in \mathcal{S}(c)$ we have
\begin{equation}\label{c3.1}
I(u)\geq\frac{1}{2}\|\nabla u\|_2^2-C(N,\mu)\|u\|_{\frac{2Np}{2N-\mu}}^{2p}\geq \frac{1}{2}\|\nabla u\|_2^2-C(N,\mu,p)c^{\frac{2N-\mu-(N-2)p}{2}}\|\nabla u\|^{Np-2N+\mu}_{2}.
\end{equation}
Since $$p<\bar{p}\Rightarrow Np-2N+\mu< 2,$$
thus $I$ is bounded from below on $\mathcal{S}(c)$ for any $c > 0$, and $\sigma(c)$ is well defined. For any $c> 0$, we can choose a function $u_0\in\mathcal{C}^\infty_0(\R^N, [-M, M])$ satisfying $\|u_0\|_2^2 = c$ for
some constant $M> 0$. Then it follows from $(A_1)$ and $(\ref{I2})$ that
\begin{equation}\label{c3.3}
I({u_0}^{t})\leq \frac{t^2}{2}\|\nabla u_0\|_2^2-\frac{A_\infty^2t^{Np-2N+\mu}}{2p}\int\int\frac{ |u_0(x)|^p |u_0(y)|^p}{|x-y|^\mu}dxdy,\quad \forall t\in(0,1].
\end{equation}
Since $0 < Np-2N+\mu < 2$, $(\ref{c3.3})$ implies that $I({u_0}^{t}) < 0$ for small
$t\in(0, 1)$. Jointly with the fact that $\|{u_0}^t\|_2 = \|u_0\|_2$, we obtain
 $$\sigma(c) \leq\inf\limits_{t\in(0,1]} I({u_0}^t) < 0.$$
\end{proof}

\begin{lemma}\label{lm3.2} $\sigma(c)$ is continuous on $(0, +\infty)$.
\end{lemma}
\begin{proof}
For any $c > 0$, let $c_n > 0$ and $c_n \to c$. For every $n\in\N$, let $u_n\in\mathcal{S}(c_n)$ such
that $I(u_n) < \sigma(c_n) + \frac{1}{n} < \frac{1}{n}$. Then $(\ref{c3.1})$ implies that $\{u_n\}$ is bounded in $H^1(\R^N,\R)$.
Moreover, we have
\begin{equation}\label{c3.4}
\begin{aligned} \sigma(c) & \leq I\left(\sqrt{\frac{c}{c_{n}}} u_{n}\right) \\
&=\frac{c}{2 c_{n}}\left\|\nabla u_{n}\right\|_{2}^{2}-\frac{c^{p}}{2pc_n^{p}}\int\int\frac{ A(x)|u_{n}(x)|^p A(y)|u_{n}(y)|^p}{|x-y|^\mu}dxdy \\
&=I\left(u_{n}\right)+o(1) \leq \sigma\left(c_{n}\right)+o(1). \end{aligned}
\end{equation}
On the other hand, given a minimizing sequence $\{v_n\}\subset\mathcal{S}(c)$ for $I$, we have
$$
\sigma\left(c_{n}\right) \leq I\left(\sqrt{\frac{c_{n}}{c}} v_{n}\right) \leq I\left(v_{n}\right)+o(1)=\sigma(c)+o(1),
$$
which together with $(\ref{c3.4})$, implies that $\lim\limits_{n\to+\infty} \sigma(c_n) = \sigma(c)$.
\end{proof}

From \cite{MR778970,MR778974}, we know that subadditivity inequality implies the compactness of the minimizing sequence for $\sigma(c)$ (up to translations). Although $I$ is not invariant by translations, by using the following subadditivity inequality and comparing with the limit equation we can still verify that $\sigma(c)$ has a minimizer.
\begin{lemma}\label{lm3.3} For each $c > 0$,
\begin{equation}\label{c3.5}
\sigma(c)<\sigma(\alpha)+\sigma(c-\alpha), \quad \forall 0<\alpha<c.
\end{equation}
\end{lemma}
\begin{proof}
Letting $\{u_n\}\subset\mathcal{S}(c)$ be such that $I(u_n)\to\sigma(c)$, it follows from $(\ref{c3.1})$ and
Lemma \ref{lm3.1} that $\sigma(c) < 0$, and $\{u_n\}$ is bounded in $H^1(\R^N,\R)$. Now, we claim that there
exists a constant $\rho_0 > 0$ such that
\begin{equation}\label{c3.6}
\liminf _{n \rightarrow \infty}\left\|\nabla u_{n}\right\|_{2}>\rho_{0}.
\end{equation}
Otherwise, if $(\ref{c3.6})$ is not true, then up to a subsequence, $\|\nabla u_n\|_2 \to 0$, and so $(\ref{c3.1})$
yields $$0 > \sigma(c) = \lim\limits_{n\to+\infty} I(u_n) = 0.$$
This contradiction shows that $(\ref{c3.6})$ holds.

Let ${u_n}_t=t^\varrho u_n(x/t)$ $\forall t>0$, the constant $\varrho$ is given in the condition $(A_2)$.
Then by $(A_2)$, we have
\begin{equation}\label{c3.7}
\begin{aligned}
I\left({u_n}_t\right)&=I\left(t^\varrho u_{n}(x/t)\right)  \\
&=\frac{t^{2 \varrho+N-2}}{2}\left\|\nabla u_{n}\right\|_{2}^{2}-\frac{t^{2N-\mu+2\varrho p}}{2p} \int\int\frac{A(tx)\left|u_{n}(x)\right|^pA(ty)\left|u_{n}(y)\right|^p}{|x-y|^\mu} dxdy \\
&\leq \frac{t^{2 \varrho+N-2}}{2}\left\|\nabla u_{n}\right\|_{2}^{2}-\frac{t^{2 \varrho+N}}{2p} \int\int\frac{A(x)\left|u_{n}(x)\right|^pA(y)\left|u_{n}(y)\right|^p}{|x-y|^\mu} dxdy \\
&=t^{2 \varrho+N} I\left(u_{n}\right)+\frac{t^{2 \rho+N}\left(t^{-2}-1\right)}{2}\left\|\nabla u_{n}\right\|_{2}^{2}, \quad \forall t>1.
\end{aligned}
\end{equation}
Since $\|{u_n}_t\|_2^2 = t^{2\varrho+N}\|u_{n}\|_2^2 =t^{2\varrho+N}c$ for all $t > 0$, then it follows from $(\ref{c3.6})$ and $(\ref{c3.7})$
that
\begin{equation*}
\begin{aligned} \sigma\left(t^{2\varrho+N} c\right) & \leq I\left({u_n}_{t}\right) \leq t^{2 \varrho+N} \sigma(c)+\frac{t^{2 \varrho+N}\left(t^{-2}-1\right)}{2} \rho_{0}^{2}+o(1), \quad \forall t>1, \end{aligned}
\end{equation*}
which implies
\begin{equation}\label{c3.8}
\sigma(tc) < t\sigma(c),~~\forall t > 1.
\end{equation}
Moreover, it follows from $(\ref{c3.8})$ that
$$
\sigma(c)=\frac{\alpha}{c} \sigma(c)+\frac{c-\alpha}{c} \sigma(c)<\sigma(\alpha)+\sigma(c-\alpha), \quad \forall 0<\alpha<c.
$$
This completes the proof.
\end{proof}

\begin{lemma}\label{lm3.4} $\sigma(c) \leq\sigma_\infty(c)$ for any $c > 0$.
\end{lemma}
\begin{proof}
Let $c > 0$ be given and let $\{u_n\}\subset\mathcal{S}(c)$ be such that $I_\infty(u_n)\to \sigma_\infty(c)$. Since
$A_\infty\leq A(x)$ for all $x\in\R^N$, it follows from $(\ref{I})$ that
$$\sigma(c) \leq I(u_n) \leq I_\infty(u_n) = \sigma_\infty(c) + o(1),$$
which implies that $\sigma(c)\leq\sigma_\infty(c)$ for any $c > 0$.
\end{proof}

\begin{lemma}\label{lm3.5} For each $c > 0$, $\sigma(c)$ has a minimizer.
\end{lemma}
\begin{proof}
In view of Lemma \ref{lm3.1}, we have $\sigma(c) < 0$. Let $\{u_n\}\subset\mathcal{S}(c)$ be such that
$I(u_n)\to\sigma(c)$. Then $(\ref{c3.1})$ implies that $\{u_n\}$ is bounded in $H^1(\R^N,\R)$. We then may
assume that for some $\bar{u}\in H^1(\R^N,\R)$ such that up to a subsequence, $u_n\rightharpoonup\bar{u}$ in $H^1(\R^N,\R)$.

Case (i): $\bar{u}= 0$. Then $u_n \to 0$ in $L^s_{\text{loc}}(\R^N,\R)$ for $1\leq s < 2^*$
and $u_n \to 0$ a.e. in $\R^N$. By $(A_1)$, it is easy to check that
\begin{equation}\label{c3.10}
\int\int\frac{\left(A_\infty^2-A(x)A(y)\right)|u_n(x)|^p|u_n(y)|^p}{|x-y|^\mu}dxdy\to0\quad\text{as}~n\to\infty.
\end{equation}
Then (\ref{I}), (\ref{c1.13}), and (\ref{c3.10}) imply
\begin{equation}\label{c3.11}
I_\infty(u_n)\to\sigma(c)\quad\text{as}~n\to\infty.
\end{equation}
Next, we show that
\begin{equation}\label{c3.12}
 \delta:= \limsup\limits_{n\to+\infty}
\sup\limits_{y\in\R^N}\int_{B_1(y)} |u_n|^2dx > 0.
\end{equation}
In fact, if $\delta=0$, by Lions' concentration compactness principle \cite{MR778970,MR778974},
we have $u_n \to 0$ in $L^q(\R^N,\R)$ for $2 < q < 2^*$, and so $(A_1)$ and $(A_2)$ imply that
$$\int\int\frac{A(x)A(y)|u_n(x)|^p|u_n(y)|^p}{|x-y|^\mu}dxdy\to0\quad\text{as}~n\to\infty.$$
 Then by (\ref{I}), we have
 $$0 > \sigma(c) = \lim\limits_{n\to+\infty} I(u_n) = \lim\limits_{n\to+\infty}\frac{1}{2}\|u_n\|_2^2 \geq 0,$$
which is impossible. Hence, we have $\delta> 0$, and there exists a sequence $\{y_n\}\subset\R^N$
such that
\begin{equation}\label{c3.13}
\int_{B_1(y_n)} |u_n|^2dx >\frac{\delta}{2}.
\end{equation}
Let $\hat{u}_n(x) = u_n(x + y_n)$. Then (\ref{c3.11}) leads to
\begin{equation}\label{c3.14}
\hat{u}_n \in \mathcal{S}(c),\quad I_\infty(\hat{u}_n) \to \sigma(c).
\end{equation}
In view of (\ref{c3.13}), we may assume that there exists $\hat{u}\in H^1(\R^N,\R)\setminus\{0\}$ such that, passing
to a subsequence,
\begin{equation}\label{c3.15}
\hat{u}_n\rightharpoonup \hat{u}~\text{in}~H^1(\R^N,\R), \quad\hat{u}_n \to \hat{u}~\text{in}~L^q_{\text{loc}}(\R^N,\R)~\forall q\in [1, 2^*), \quad
 \hat{u}_n\to \hat{u}~a.e.~\text{in}~\R^N.
\end{equation}
Then it follows from (\ref{c3.14}), (\ref{c3.15}), Lemmas \ref{lm3.2}, \ref{lm3.4} and \ref{lm2.11} that
\begin{equation}\label{c3.16}
\aligned
\sigma_\infty(c)\geq&\sigma(c)=\lim\limits_{n\to+\infty}I_\infty(\hat{u}_n)=I_\infty(\hat{u})+\lim\limits_{n\to+\infty}I_\infty(\hat{u}_n-\hat{u})\\
\geq&\sigma_\infty(\|\hat{u}\|_2^2)+\lim\limits_{n\to+\infty}\sigma_\infty(\|\hat{u}_n-\hat{u}\|_2^2)=\sigma_\infty(\|\hat{u}\|_2^2)+\sigma_\infty(c-\|\hat{u}\|_2^2).
\endaligned
\end{equation}
If $\|\hat{u}\|_2^2 < c$, then (\ref{c3.16}) and Lemma \ref{lm3.3} imply
 $$\sigma_\infty(c) \geq\sigma_\infty(\|\hat{u}\|_2^2)+\sigma_\infty(c-\|\hat{u}\|_2^2) >\sigma_\infty(c),$$
which is impossible. This shows $\|\hat{u}\|_2^2  = c$. Then we have $\hat{u}_n \to\hat{u}$ in $L^q(\R^N,\R)$
for $2 \leq q < 2^*$. From this, the weak semicontinuity of norm and (\ref{c3.16}), we derive
$$\sigma(c) = \lim\limits_{n\to+\infty} I(\hat{u}_n)\geq I(\hat{u}) \geq\sigma(c),$$
which leads to $\sigma(c) = I(\hat{u})$. Hence, $\hat{u}$ is a minimizer of $\sigma(c)$ for any $c > 0$.

Case (ii): $\bar{u}\neq0$. Then $u_n \to \bar{u}$ in $L^q_{\text{loc}}(\R^N,\R)$ for $1\leq q< 2^*$
and $u_n \to \bar{u}$ a.e. in $\R^N$. By Lemmas \ref{lm3.2} and \ref{lm2.11}, we have
\begin{equation}\label{c3.9}
\aligned
\sigma(c)=&\lim\limits_{n\to+\infty}I(\bar{u}_n)=I(\bar{u})+\lim\limits_{n\to+\infty}I(\bar{u}_n-\bar{u})\\
\geq&\sigma(\|\bar{u}\|_2^2)+\lim\limits_{n\to+\infty}\sigma(\|\bar{u}_n-\bar{u}\|_2^2)=\sigma(\|\bar{u}\|_2^2)+\sigma(c-\|\bar{u}\|_2^2).
\endaligned
\end{equation}
If $\|\bar{u}\|_2^2 < c$, then (\ref{c3.9}) and Lemma \ref{lm3.3} imply
 $$\sigma(c) \geq\sigma(\|\bar{u}\|_2^2)+\sigma(c-\|\bar{u}\|_2^2) >\sigma(c),$$
which is impossible. This shows $\|\bar{u}\|_2^2  = c$. Then we have $\bar{u}_n \to\bar{u}$ in $L^q(\R^N,\R)$
for $2 \leq q < 2^*$. From this, the weak semicontinuity of norm and (\ref{c3.16}), we have
$$\sigma(c) = \lim\limits_{n\to+\infty} I(\bar{u}_n)\geq I(\bar{u}) \geq\sigma(c),$$
which leads to $\sigma(c) = I(\bar{u})$. Hence, $\bar{u}$ is a minimizer of $\sigma(c)$ for any $c > 0$.
\end{proof}

{\it Proof of Theorem \ref{th1.1}.} For any $c > 0$, from Lemma \ref{lm3.5}, there exists $\bar{u}_c \in \mathcal{S}(c)$ such
that $I(\bar{u}_c) = \sigma(c)$. In view of the Lagrange multiplier theorem, there exists $\lambda_c \in\R$
such that
$$I'(\bar{u}_c) = \lambda_c\bar{u}_c.$$
Therefore, $(\bar{u}_{c},\lambda_c)$ is a solution of (\ref{1.1.0}).
\qed

\vskip4mm
{\subsection{ The proof of Theorem \ref{th1.1.5} }}
In this section,  under condition $(A'_1)$ and $p\in(2, \bar{p})$,  we prove Theorem \ref{th1.1.5} by using the following abstract variational principle \cite[Proposition 1.2]{bellazzini}.
\begin{proposition}\label{pr1.2}(\cite[Proposition 1.2]{bellazzini})
Let $\mathcal{H}$, $\mathcal{H}_1$ and $\mathcal{H}_2$ be three Hilbert spaces such that
$$\mathcal{H}\subset\mathcal{H}_1,\quad \mathcal{H}\subset\mathcal{H}_2$$
and
$$C_1\left(\|u\|^2_{\mathcal{H}_1}+\|u\|^2_{\mathcal{H}_2}\right)\leq\|u\|^2_{\mathcal{H}}\leq C_2\left(\|u\|^2_{\mathcal{H}_1}+\|u\|^2_{\mathcal{H}_2}\right)~~\forall u\in\mathcal{H}.$$
For given $c>0$, let $W, T: \mathcal{H}\mapsto \R$ such that:\\
(1) $T(0)=0$;\\
(2) $T$ is weakly continuous;\\
(3) $T(\nu u)\leq\nu^2 T(u)$ and $W(\nu u)\leq \nu^2 W(u)$, $\forall\nu\geq1,~u\in \mathcal{H}$;\\
(4) If $u_n\rightharpoonup u$ in $\mathcal{H}$ and $u_n\to u$ in $\mathcal{H}_2$, then $W(u_n)\to W(u)$;\\
(5) If $u_n\rightharpoonup u$ in $\mathcal{H}$, then $W(u_n-u)+W(u)=W(u_n)+o(1)$;\\
(6) $-\infty<\varsigma^{W+T}(c)<\varsigma^{W}(c)$, where
\begin{equation}\label{J1.8}
\aligned
&\varsigma^{W+T}(c):=\inf\limits_{u\in B_{\mathcal{H}_2}(c)\cap\mathcal{H}}\left(\frac{1}{2}\| u\|^2_{\mathcal{H}_1}+W(u)+T(u)\right),
\endaligned
\end{equation}
$$\varsigma^{W}(c):=\inf\limits_{u\in B_{\mathcal{H}_2}(c)\cap\mathcal{H}}\left(\frac{1}{2}\| u\|^2_{\mathcal{H}_1}+W(u)\right),$$
and
$$B_{\mathcal{H}_2}(c):=\{u\in\mathcal{H}_2: \|u\|_{\mathcal{H}_2}^2=c\};$$
(7) For every sequence $\{u_n\}\subset B_{\mathcal{H}_2}(c)\cap\mathcal{H}$ such that $\|u_n\|_{\mathcal{H}}\to\infty$, we have
$$\frac{1}{2}\|u_n\|_{\mathcal{H}_1}^2+W(u_n)+T(u_n)\to\infty~\text{as}~n\to\infty.$$
Then every minimizing sequence for (\ref{J1.8}), i.e.,
$$u_n\in B_{\mathcal{H}_2}(c)\cap\mathcal{H}~~\text{and}~~\frac{1}{2}\|u_n\|^2_{\mathcal{H}_1}+W(u_n)+T(u_n)\to \varsigma^{W+T}(c),$$
is compact in $\mathcal{H}$.
\end{proposition}

\begin{lemma}\label{J3.1} Assume that $A(x)$, $A_0$ and $p$ are as in Theorem \ref{th1.1.5}. Then there exists $c_0>0$ such that:
$$\sigma^{A(x),A(y)}(c)<\sigma^{\min\{A(x),A_0\},A(y)}(c)~~\forall~c>c_0,$$
where
$$\sigma^{A(x),A(y)}(c):=\inf\limits_{u\in \mathcal{S}(c)}\left[\frac{1}{2} \int|\nabla u|^{2} \mathrm{d} x-\frac{1}{2p}\int\int\frac{ A(x) |u(x)|^pA(y) |u(y)|^p}{|x-y|^\mu} \mathrm{d} x\mathrm{d} y\right]=\sigma(c)$$
and
$$\sigma^{\min\{A(x),A_0\},A(y)}(c):=\inf\limits_{u\in \mathcal{S}(c)}\left[\frac{1}{2} \int|\nabla u|^{2} \mathrm{d} x-\frac{1}{2p}\int\int\frac{ \min\{A(x),A_0\}|u(x)|^pA_0|u(y)|^p}{|x-y|^\mu} \mathrm{d} x\mathrm{d} y\right].$$
\end{lemma}
\begin{proof}
Since $A\in L^\infty(\R^N,\R)$, then by Lebesgue derivation Theorem that
$$\lim\limits_{\delta\to0}\delta^{-N}\int_{B_\delta(x_0)}|A(x)-A(x_0)|^{\frac{2N}{2N-\mu}}dx=0~~\text{for~almost~each~}x_0\in\R^N.$$
By condition $(A'_1)$ we have $meas\{x\in\R^N: A(x)>A_0\}>0$, then we deduce that there exists $\tilde{x}\in\{x\in\R^N: A(x)>A_0\}$ such that
$$\lim\limits_{\delta\to0}\delta^{-N}\int_{B_\delta(\tilde{x})}|A(x)-A(\tilde{x})|^{\frac{2N}{2N-\mu}}dx=0.$$
For simplicity we can assume that $\tilde{x}\equiv0$, hence we have
\begin{equation}\label{J3.2}
\lim\limits_{\delta\to0}\delta^{-N}\int_{B_\delta(0)}|A(x)-A(0)|^{\frac{2N}{2N-\mu}}dx=0~\text{with}~A(0)>A_0.
\end{equation}
By Lemma \ref{lm3.5}, there exists a minimizer $u_0\in H^1(\R^N,\R)$ for $\sigma^{A_0}(1)$. It is easy to check that
\begin{equation}\label{J3.3}
u_0^c:=u_0(x/c^{a})c^{-b}~\text{is~a~minimizer~for~}\sigma^{A_0}(c),
\end{equation}
where
$$a:=\frac{p-1}{N(p-2)+\mu-2},\quad b:=\frac{N+2-\mu}{2N(p-2)+2\mu-4}.$$
Notice that by $p<\bar{p}=2+\frac{2-\mu}{N}$, we have
$$a<0~~\text{and}~~b<0.$$
We claim that there is $c_0>0$ such that
\begin{equation}\label{J3.5}
\sigma^{A(x),A(y)}(c)<\sigma^{A_0, A(y)}(c)~\forall~c>c_0.
\end{equation}
On the other hand $0\leq\min\{A(x),A_0\}\leq A_0$ implies that
\begin{equation}\label{J3.6}
\sigma^{A_0, A(y)}(c)\leq\sigma^{A(x), A(y)}(c).
\end{equation}
By combining (\ref{J3.5}) and (\ref{J3.6}) we get the desired result.

Next we prove (\ref{J3.5}). Due to (\ref{J3.3}) it is sufficient to prove the following inequality:
\begin{equation*}
\aligned
&\frac{1}{2}\|\nabla u_0^c\|_2^2-\frac{1}{2p}\int\int\frac{A(x)|u_0^c(x)|^pA(y)|u_0^c(y)|^p}{|x-y|^{\mu}}dxdy\\
<&\frac{1}{2}\|\nabla u_0^c\|_2^2-\frac{A_0}{2p}\int\int\frac{|u_0^c(x)|^pA(y)|u_0^c(y)|^p}{|x-y|^{\mu}}dxdy
\endaligned
\end{equation*}
or equivalently
\begin{equation}\label{J3.7}
\aligned
&I(c)+II(c):= \\
&\frac{A_0-A(0)}{2p}\int\int\frac{|u_0^c(x)|^p|u_0^c(y)|^p}{|x-y|^{\mu}}dxdy
+\frac{1}{2p}\int\int\frac{(A(0)-A(x))|u_0^c(x)|^pA(y)|u_0^c(y)|^p}{|x-y|^{\mu}}dxdy\\
<&0.
\endaligned
\end{equation}
By $A_0<A(0)$ we can fix $R_0>0$ such that
\begin{equation}\label{J3.8}
\aligned
&\frac{A_0-A(0)}{2p}\int\int\frac{|u_0(x)|^p|u_0(y)|^p}{|x-y|^{\mu}}dxdy \\
&+\frac{1}{2^{\frac{\mu}{2N}}p}\|A\|_\infty^{2-\frac{\mu}{2N}}\|u_0\|_{\frac{2Np}{2N-\mu}}^p\left( \int_{|x|\geq R_0 }
|u_0(x)|^{\frac{2Np}{2N-\mu}}dx \right)^{\frac{2N-\mu}{2N}}\\
=&-\varepsilon_0<0.
\endaligned
\end{equation}
By calculation, we get
$$I(c)=\frac{A_0-A(0)}{2p}c^{(2N-\mu)a-2pb}\int\int\frac{|u_0(x)|^p|u_0(y)|^p}{|x-y|^{\mu}}dxdy.$$
And by (\ref{J3.2}) and Hardy-Littlewood-Sobolev inequality we have
\begin{equation*}
\aligned
II(c)=&\frac{1}{2p}\int\int\frac{(A(0)-A(x))|u_0^c(x)|^pA(y)|u_0^c(y)|^p}{|x-y|^{\mu}}dxdy\\
\leq& \frac{\|A\|_\infty}{2p}\|u^c_0\|_{\frac{2Np}{2N-\mu}}^p \left(\int|A(0)-A(x)|^{\frac{2N}{2N-\mu}}|u_0^c(x)|^{\frac{2Np}{2N-\mu}}dx\right)^{\frac{2N-\mu}{2N}}
\\
\leq& \frac{\|A\|_\infty}{2p}c^{\frac{(2N-\mu)a-2pb}{2}}\|u_0\|_{\frac{2Np}{2N-\mu}}^p \\
&\cdot\left(2\|A\|_\infty c^{Na-\frac{2Np}{2N-\mu}b}\int_{|x|\geq R_0 }
|u_0(x)|^{\frac{2Np}{2N-\mu}}dx \right.\\
&\left.+c^{Na-pb}\|u_0\|_\infty c^{-Na}\int_{|x|\leq R_0 c^a}|A(0)-A(x)|^{\frac{2N}{2N-\mu}}dx\right)^{\frac{2N-\mu}{2N}}\\
=&\frac{\|A\|_\infty}{2p}\|u_0\|_{\frac{2Np}{2N-\mu}}^pc^{(2N-\mu)a-2pb}\left(2\|A\|_\infty \int_{|x|\geq R_0 }
|u_0(x)|^{\frac{2Np}{2N-\mu}}dx \right)^{\frac{2N-\mu}{2N}}\\
&+c^{(2N-\mu)a-2pb+\frac{\mu}{2N}pb}o(1),
\endaligned
\end{equation*}
where $\lim\limits_{c\to\infty}o(1)=0$.
Therefore, by (\ref{J3.8}) we have
\begin{equation*}
\aligned
I(c)+II(c)\leq&c^{(2N-\mu)a-2pb}\left[\frac{A_0-A(0)}{2p}\int\int\frac{|u_0(x)|^p|u_0(y)|^p}{|x-y|^{\mu}}dxdy \right.\\
&\left.+\frac{1}{2^{\frac{\mu}{2N}}p}\|A\|_\infty^{2-\frac{\mu}{2N}}\|u_0\|_{\frac{2Np}{2N-\mu}}^p\left( \int_{|x|\geq R_0 }
|u_0(x)|^{\frac{2Np}{2N-\mu}}dx \right)^{\frac{2N-\mu}{2N}}+c^{\frac{\mu}{2N}pb}o(1)\right]\\
&=c^{(2N-\mu)a-2pb}(-\varepsilon_0+o(1)),
\endaligned
\end{equation*}
which implies (\ref{J3.7}) for $c$ large enough, and in turn it is equivalent (\ref{J3.5}).
\end{proof}

 {\it Proof of Theorem \ref{th1.1.5}.} It is sufficient to show that any minimizing sequence for $\sigma(c)=\sigma^{A(x),A(y)}(c)$ is compact in $H^1(\R^N,\R)$. We use Proposition \ref{pr1.2} to prove this by choosing
 \begin{equation*}
\aligned
&\mathcal{H}=H^1(\R^N,\R),\quad \mathcal{H}_1=\mathcal{D}^{1,2}(\R^N,\R),\quad \mathcal{H}_2=L^2(\R^N,\R),\\
&W(u)=-\frac{1}{2p}\int_{\R^N}\int_{\R^N}\frac{\min\{A(x),A_0\}|u(x)|^pA(y)|u(y)|^p}{|x-y|^\mu}dxdy,\\
&T(u)=-\frac{1}{2p}\int_{A(x)\geq A_0}\int_{y\in\R^N}\frac{(A(x)-A_0)|u(x)|^pA(y)|u(y)|^p}{|x-y|^\mu}dxdy.
 \endaligned
\end{equation*}
Then
 \begin{equation*}
\aligned
W(u)+T(u)=-\frac{1}{2p}\int_{\R^N}\int_{\R^N}\frac{A(x)|u(x)|^pA(y)|u(y)|^p}{|x-y|^{\mu}}dxdy
 \endaligned
\end{equation*}
and $\varsigma^{W+T}(c)=\sigma(c)$.

It is easy to verify that the conditions (1), (3) in Proposition \ref{pr1.2} hold. The left hand side inequality in (6) follows from Lemma \ref{lm3.1}; The right hand side inequality in (6) follows from Lemma \ref{J3.1} provided that $c>c_0$,
where $c_0$ comes from (\ref{J3.5}). Since
$$p<\bar{p}\Rightarrow 2>Np-2N+\mu,$$
then (\ref{c3.1}) implies (7). (5) follows from Lemma \ref{lm2.11}.

Next, we prove (2). Let $u_n\rightharpoonup u$ in $H^1(\R^N,\R)$, then $\{u_n-u\}$ is bounded in $H^1(\R^N,\R)$.
By $(A'_1)$, $\{x\in\R^N: A(x)\geq A_0\}$ is bounded in $\R^N$, and thus by Rellich Compactness Theorem
$$\int_{A(x)\geq A_0}|u_n(x)-u(x)|^{\frac{2Np}{2N-\mu}}dx\to0.$$
Then, by Lemma \ref{lm2.11} (Brezis-Lieb lemma of nonlocal version) and Hardy-Littlewood-Sobolev inequality, we have
 \begin{equation*}
\aligned
|T(u_n)-T(u)|=&\frac{1}{2p}\left|\int_{A(x)\geq A_0}\int_{y\in\R^N}\frac{(A(x)-A_0)A(y)(|u_n(x)|^p|u_n(y)|^p-|u(x)|^p|u(y)|^p)}{|x-y|^\mu}dxdy\right|\\
\leq&\frac{\|A\|_\infty^2}{p}\left|\int_{A(x)\geq A_0}\int_{y\in\R^N}\frac{(|u_n(x)|^p|u_n(y)|^p-|u(x)|^p|u(y)|^p)}{|x-y|^\mu}dxdy\right|\\
=&\frac{\|A\|_\infty^2}{p}\left|\int_{A(x)\geq A_0}\int_{y\in\R^N}\frac{|u_n(x)-u(x)|^p|u_n(y)-u(y)|^p}{|x-y|^\mu}dxdy+o(1)\right|\\
\leq&C\left|\int_{A(x)\geq A_0}|u_n(x)-u(x)|^{\frac{2Np}{2N-\mu}}dx\right|^{\frac{2N-\mu}{2N}}\|u_n-u\|^p_{\frac{2Np}{2N-\mu}}+o(1)\\
&\to 0.
 \endaligned
\end{equation*}
Thus $T$ is weakly continuous, i.e., (2) holds.

Finally we shall verify (4). Let $u_n\rightharpoonup u$ in $H^1(\R^N,\R)$ and $u_n\to u$ in $L^2(\R^N,\R)$. Then we have
$$u_n\to u~~\text{in}~L^q(\R^N,\R)~~\forall 2\leq q<2^*.$$
Then, Lemma \ref{lm2.11} and Hardy-Littlewood-Sobolev inequality imply
 \begin{equation*}
\aligned
|W(u_n)-W(u)|
\leq&\frac{\|A\|_\infty^2}{2p}\left|\int_{\R^N}\int_{\R^N}\frac{(|u_n(x)|^p|u_n(y)|^p-|u(x)|^p|u(y)|^p)}{|x-y|^\mu}dxdy\right|\\
\leq&C\|u_n-u\|^{2p}_{\frac{2Np}{2N-\mu}}+o(1)\to 0.
 \endaligned
\end{equation*}
Thus (4) holds.
 \qed

\vskip4mm
{\section{ $L^2-$supercritical case }}
 \setcounter{equation}{0}
 \vskip4mm
We prove Theorem \ref{th1.2} in this section. Our method is derived from \cite{MR4081327}. $(A_1)$, $(A_3)$ and $(A_4)$ hold with $\bar{p}<p<2^*_\mu$.
Since $A(x) \equiv A_\infty$ satisfies $(A_1)$, $(A_3)$ and $(A_4)$, all the following conclusions on $I$ are also true for $I_\infty$.
\begin{lemma}\label{lm2.1} We have
\begin{equation}\label{c2.3}
\begin{array}{ll}
\psi(t,x):=&-2 t^{-\frac{Np-2N+\mu}{2}}[A(x)-A(t x)]\\
&+\frac{4(t^{-\frac{Np-2N+\mu}{2}}-1)}{Np-2N+\mu} \nabla A(x) \cdot x\geq 0,\quad~ \forall  t>0, x\in \mathbb{R}^{N};
\end{array}
\end{equation}
\begin{equation}\label{c2.4}
t\mapsto A(t x)~\text{is~nonincreasing~on}~(0,\infty), \quad \forall x \in \mathbb{R}^{N};
 \end{equation}
 \begin{equation}\label{c2.5}
-\nabla A(x) \cdot x \geq 0,\quad\forall x \in \mathbb{R}^{N},  \text { and } -\nabla A(x) \cdot x \rightarrow 0, \quad \text { as }|x| \rightarrow \infty.
\end{equation}
\end{lemma}
\begin{proof}
First, for any $x\in \R^N$, by $(A_3)$, we have
\begin{equation*}
\aligned
\frac{d\psi(t, x)}{dt}
=&t^{-1-\frac{Np-2N+\mu}{2}}\big\{[(Np-2N+\mu)A(x)-2\nabla A(x)\cdot x]\\
&-[(Np-2N+\mu)A(tx)-2\nabla A(tx)\cdot (tx)]\big\}
\\
&\quad\quad \left\{
\begin{array}{ll}
\aligned
&\geq0,~~~t\geq1,\\
&\leq0,~~~0<t<1,
\endaligned
\end{array}
\right.
\endaligned
\end{equation*}
which implies that $\psi(t, x) \geq \psi(1, x) = 0$ for all $t > 0$ and $x\in\R^N$ , i.e., $(\ref{c2.3})$ holds.

Next, let $t \to+\infty$ in $(\ref{c2.3})$, we have $-\nabla A(x)\cdot x \geq 0$ for all $x\in\R^N$, which leads to $(\ref{c2.4})$.
Last, let $t = 1/2$ in $(\ref{c2.3})$, then one has
$$0\leq-\nabla A(x)\cdot x\leq -\frac{2^{\frac{Np-2N+\mu}{2}}(Np-2N+\mu)[A(x)-A(x/2)] }{2(2^{\frac{Np-2N+\mu}{2}}-1)}\to0,\quad \text{as}~|x|\to+\infty.$$
This shows $(\ref{c2.5})$ holds.
\end{proof}

\begin{lemma}\label{cr2.7} For $u \in \mathcal{M}(c)$,
$I(u)>I\left( u^{t}\right)$ for all $t\in(0,1)\cup(1,+\infty)$, where $u^t(x) = t^{\frac{N}{2}} u(tx)$.
\end{lemma}
\begin{proof}
By $p\in\left(\frac{2N-\mu+2}{N}, \frac{2N-\mu}{(N-2)_{+}}\right)$, $(A_1)$ and $(\ref{c2.5})$, for $u \in \mathcal{M}(c)$, we have
 \begin{equation*}
\aligned
I(u)=&I(u)-\frac{1}{2}J(u) \\
=&\frac{1}{4p}\int\int\frac{ \left[(Np-2N+\mu-2)A(x)-2\nabla A(x)\cdot x\right]A(y)|u(x)|^p  |u(y)|^p}{|x-y|^\mu}dxdy
>0.
\endaligned
\end{equation*}
Fix a $u \in \mathcal{M}(c)$, let
 \begin{equation}\label{g}
\aligned
g(t,u):=&I(u)-I(u^t)
=\frac{1-t^2}{2}\|\nabla u\|_2^2-\frac{1}{2p}\int\int\frac{ A(x)|u(x)|^p A(y) |u(y)|^p}{|x-y|^\mu}dxdy\\
&+\frac{t^{Np-2N+\mu}}{2p}\int\int\frac{ A(t^{-1}x)|u(x)|^p A(t^{-1}y) |u(y)|^p}{|x-y|^\mu}dxdy.
\endaligned
\end{equation}
Then we have
\begin{equation}\label{2.14}
g(0,u)=I(u)>0,\quad g(1,u)=0,\quad g(+\infty,u)=+\infty.
\end{equation}
By $(\ref{I2})$, we have
 \begin{equation*}
\aligned
\frac{dg(t,u)}{dt}=-t\left(\|\nabla u\|_2^2-h(t,u)\right),\quad \frac{dg(t,u)}{dt}|_{t=1}=-J(u)=0,
\endaligned
\end{equation*}
where
 \begin{equation*}
 \aligned
 h(t,u)=&\frac{t^{Np-2N+\mu-2}}{2p} \cdot \\
 &\int\int\frac{\left[(Np-2N+\mu)A(t^{-1}x)-2\left(\nabla A(t^{-1}x)\cdot (t^{-1}x)\right)\right] A(t^{-1}y)|u(x)|^p |u(y)|^p}{|x-y|^\mu}dxdy.
 \endaligned
\end{equation*}
Using $(A_1)$, $(A_3)$ and Lemma \ref{lm2.1}, we have
 \begin{equation}\label{2.15}
 \aligned
 &h(0,u)=0;\quad h(t,u)\to+\infty~\text{as}~t\to+\infty; \\
\text{~the~function}~&t\mapsto h(t,u)~\text{strictly~increasing~on} ~(0,+\infty).
 \endaligned
 \end{equation}
Then by $(\ref{2.15})$, $t=1$ is the unique solution of equation $\frac{dg(t,u)}{dt}=0$. This together with $(\ref{2.14})$ implies the conclusion.
\end{proof}

\begin{lemma}\label{lm2.8} For any $u \in \mathcal{S}(c)$, there
exists a unique $t_u > 0$ such that $ u^{t_u} \in\mathcal{M}(c)$.
\end{lemma}
\begin{proof}
Let $u\in \mathcal{S}(c)$ be fixed and define a function $\zeta(t):= I(u^t)$ on $(0, \infty)$.
  By $(\ref{J})$, we have
 \begin{equation}\label{2.7.1}
\aligned
J(u^t)=&\|\nabla u\|_2^2-t^2h(t,u).
\endaligned
\end{equation}
Clearly, by $(\ref{I2})$ and $(\ref{2.7.1})$, we have
 \begin{equation*}
 \aligned
 &\zeta'(t)=0
 &\Longleftrightarrow J(u^t)=0~  \Longleftrightarrow~ u^t\in\mathcal{M}.
 \endaligned
\end{equation*}
It is easy to verify that $\lim\limits_{t\to0}\zeta(t) = 0$,
$\zeta(t) > 0$ for $t > 0$ small and $\zeta(t) < 0$ for $t$ large. Therefore $\max\limits_{t\in[0,\infty)}\zeta(t)$ is achieved at
some $t_u > 0$ so that $\zeta'(t_u) = 0$ and $u^{t_u}\in\mathcal{M}$.
And we have from $(\ref{2.15})$ and $(\ref{2.7.1})$ that $t_u$ is unique for any $u\in S(c)$.
\end{proof}

Combining Lemma \ref{cr2.7} and Lemma \ref{lm2.8}, we get the following result.
\begin{lemma}\label{lm2.10}
$$
\inf _{u \in \mathcal{M}(c)} I(u)=m(c)=\inf _{u \in \mathcal{S}(c)} \max _{t>0} I\left( u^{t}\right).
$$
\end{lemma}

\begin{lemma}\label{lm2.12} The function $c\mapsto m(c)$ is
nonincreasing on $(0, \infty)$. In particular, if $m(c)$ is achieved, then $m(c) > m(c')$ for
any $c'> c$.
\end{lemma}
\begin{proof}
For any $c_2 > c_1 > 0$, it follows that there exists $\{u_n\}\subset\mathcal{M}(c_1)$ such that
$$I(u_n)<m(c_1)+\frac{1}{n}.$$
Let $\xi := \sqrt{c_2/c_1} \in (1,+\infty)$ and $v_n(x) := \xi^{(2-N)/2}u_n(\xi^{-1}x)$. Then $\|v_n\|_2^2 = c_2$ and
$\|\nabla v_n\|_2 = \|\nabla u_n\|_2$. By Lemma \ref{lm2.8}, there exists $t_n > 0$ such that $v_n^{t_n}\in\mathcal{M}(c_2)$.
Then it follows from $(A_4)$, $(\ref{I2})$, and Lemma \ref{cr2.7} that
\begin{equation*}
\aligned
m\left(c_{2}\right) \leq& I\left(v_{n}^{t_{n}}\right)
=I\left(u_{n}^{t_{n}}\right)+\frac{t_{n}^{Np-2N+\mu}}{2p}\cdot \\
 &\int\int\frac{\left[A\left(t_{n}^{-1} x\right)A\left(t_{n}^{-1} y\right) -\xi^{(2-N)p+2N-\mu} A\left(\xi t_{n}^{-1} x\right)A\left(\xi t_{n}^{-1} y\right)\right]|u_{n}(x)|^{p}|u_{n}(y)|^{p}}{|x-y|^\mu}dxdy \\
 \leq& I\left(u_{n}^{t_{n}}\right)
 \leq I\left(u_{n}\right)<m\left(c_{1}\right)+\frac{1}{n},
 \endaligned
\end{equation*}
which shows that $m(c_2) \leq m(c_1)$ by letting $n \to+\infty$.

Next, we assume that $m(c)$ is achieved, i.e., there exists $\tilde{u}\in\mathcal{M}(c)$ such that $I(\tilde{u})=m(c)$. For any given $c' > c$. Let $\tilde{\xi}= c'/c \in (1,+\infty)$ and $\tilde{v}(x):= \tilde{\xi}^{(2-N)/2}
\tilde{u}(\tilde{\xi}^{-1}x)$.
Then $\|\tilde{u}\|_2^2 = c'$ and $\|\tilde{u}\|_2^2= \|\tilde{v}\|_2^2$. By Lemma \ref{lm2.8}, there exists $t_0 > 0$ such that
$\tilde{v}^{t_0} \in \mathcal{M}(c')$. Then it follows from $(A_4)$, $(\ref{I2})$, and Lemma 3.2 that
\begin{equation*}
\aligned
m\left(c'\right) \leq& I\left(\tilde{v}^{t_{0}}\right)
=I\left(\tilde{u}^{t_{0}}\right)+\frac{t_{0}^{Np-2N+\mu}}{2p}\cdot \\
 &\int\int\frac{\left[A\left(t_{0}^{-1} x\right)A\left(t_{0}^{-1} y\right) -\tilde{\xi}^{(2-N)p+2N-\mu} A\left(\tilde{\xi} t_{0}^{-1} x\right)A\left(\tilde{\xi} t_{0}^{-1} y\right)\right]|\tilde{u}(x)|^{p}|\tilde{u}(y)|^{p}}{|x-y|^\mu}dxdy \\
 <& I\left(\tilde{u}^{t_{0}}\right)\leq
  I\left(\tilde{u}\right)=m\left(c\right),
 \endaligned
\end{equation*}
which shows that $m(c') < m(c)$.

\end{proof}

\begin{lemma}\label{lm2.13}
(i) There exists $\rho_0 > 0$ such that $\|\nabla u\|_2\geq \rho_0$, $\forall u\in\mathcal{M}(c)$; \\
(ii) $m(c) = \inf\limits_{u\in\mathcal{M}(c)}I(u) > 0$.
\end{lemma}
\begin{proof}
(i) For $u\in\mathcal{M}(c)$,
$$\|\nabla u\|_2^2=\frac{1}{2p}\int\int\frac{ \left[(Np-2N+\mu)A(x)-2\nabla A(x)\cdot x\right]A(y) |u(x)|^p |u(y)|^p}{|x-y|^\mu}dxdy.$$
By $(A_1)$, Lemma 3.1, and Hardy--Littlewood--Sobolev inequality,
$$\|\nabla u\|_2^2\leq C\int\int\frac{ |u(x)|^p |u(y)|^p}{|x-y|^\mu} \mathrm{d} x\mathrm{d} y \leq C(N,\mu)\|u\|_{\frac{2Np}{2N-\mu}}^{2p}.$$
On the other hand, we have from Gagliardo-Nirenberg inequality that
 \begin{equation*}
 \aligned
\| u\|_s\leq C(N,s)\|\nabla u\|_2^{\frac{N(s-2)}{2s}}\|u\|_2^{\frac{2N-(N-2)s}{2s}},\quad\forall u\in H^1(\R^N,\R),\quad s\in\left[2,2^*\right).
\endaligned
\end{equation*}
Therefore,
\begin{equation}\label{cc}
\|\nabla u\|_2^{Np-2N+\mu-2}\geq C(N,p,\mu) \|u\|_2^{Np-2N+\mu-2p}=C(N,p,\mu)c^{\frac{Np-2N+\mu-2p}{2}}.
\end{equation}
Since $Np-2N+\mu>2$, there exists $$\rho_0=C(N,p,\mu)^{\frac{1}{Np-2N+\mu-2}}c^{\frac{Np-2N+\mu-2p}{2(Np-2N+\mu-2)}} > 0$$ such that $\|\nabla u\|_2\geq\rho_0$.
\vskip2mm

(ii) For $u\in\mathcal{M}(c)$, it follows from $-\left(\nabla A(x)\cdot x\right)A(y)\geq0$ that
 \begin{equation}\label{2.10.ii}
 \aligned
I(u)&=I(u)-\frac{1}{Np-2N+\mu}J(u)\\
&=\left(\frac{1}{2}-\frac{1}{Np-2N+\mu}\right)\|\nabla u\|_2^2-\frac{1}{(Np-2N+\mu)p}\int\int\frac{\left(\nabla A(x)\cdot x\right)A(y)|u(x)|^p |u(y)|^p}{|x-y|^\mu}dxdy  \\
&\geq \left(\frac{1}{2}-\frac{1}{Np-2N+\mu}\right)\rho_0^2.
\endaligned
\end{equation}
Therefore $m(c) = \inf\limits_{u\in\mathcal{M}(c)}I(u) > 0$.
\end{proof}

By Lemma \ref{lm2.10}, we have
\begin{equation}\label{A2}
m(c) \leq m_\infty(c).
\end{equation}
With the help of (\ref{A2}), we can show the following lemma.
\begin{lemma}\label{lm2.15}  $m(c)$ is achieved.
\end{lemma}
\begin{proof}
By  Lemmas \ref{lm2.8} and \ref{lm2.13}, we have we have $\mathcal{M}(c)\neq\emptyset$ and $m(c)> 0$.
Let $\{u_n\}\subset \mathcal{M}(c)$ be such that $I(u_n)\to m(c)$. Since $J(u_n) = 0$, it follows from $(\ref{2.10.ii})$ that
 \begin{equation*}
 m(c)+o(1)=I(u_n)\geq \left(\frac{1}{2}-\frac{1}{Np-2N+\mu}\right)\|\nabla u_n\|_2^2.
 \end{equation*}
 This shows that $\{\|\nabla u_n\|_2\}$ is bounded. Passing to a subsequence, we have $$u_n\rightharpoonup \bar{u}\text{ in }
H^1(\R^N,\R),\quad
u_n \to\bar{u}\text{ in }L^s_{\text{loc}}(\R^N,\R)\text{ for }2 \leq s < 2^*,\quad\text{and }u_n\to\bar{u} ~a.e.\text{ in }\R^N.$$

 Case (i) $\bar{u}= 0$. Let $B_R(0)$ be a ball in $\R^N$ with the origin as its center and $R$ as its radius, by Lemma 2.1 and Hardy--Littlewood--Sobolev inequality, we have
\begin{equation}\label{1111}
\aligned
&\left|\int\int\frac{(\nabla A(x)\cdot x)A(y)|u_n(x)|^p|u_n(y)|^p}{|x-y|^\mu}dxdy\right|\\
=&\left|\int\left(\int_{x\in B_R(0)}+\int_{x\in B^c_R(0)}\right)\frac{(\nabla A(x)\cdot x)A(y)|u_n(x)|^p|u_n(y)|^p}{|x-y|^\mu}dxdy\right|\\
\leq&\left[\left(\int_{x\in B_R(0)}+\int_{x\in B^c_R(0)}\right)\left|\nabla A(x)\cdot x\right|^{\frac{2N}{2N-\mu}}|u_n(x)|^{\frac{2Np}{2N-\mu}}dx\right]^{\frac{2N-\mu}{2N}} \\
&\cdot\left[\int |A(y)|^{\frac{2N}{2N-\mu}}|u_n(y)|^{\frac{2Np}{2N-\mu}}dy\right]^{\frac{2N-\mu}{2N}}\\
\to&0,\quad \text{as}~R\to\infty,~n\to\infty.
\endaligned
\end{equation}
Similarly, by $\lim\limits_{|x|\to\infty}A(x)=A_\infty$, we have
 \begin{equation}\label{2222}
 \aligned
&\int\int\frac{(A_\infty^2-A(x)A(y))|u_n(x)|^p|u_n(y)|^p}{|x-y|^\mu}dxdy \\
=&\int\int\frac{A_\infty(A_\infty-A(x))|u_n(x)|^p|u_n(y)|^p}{|x-y|^\mu}dxdy+\int\int\frac{A(x)(A_\infty-A(y))|u_n(x)|^p|u_n(y)|^p}{|x-y|^\mu}dxdy\\
=&\int\left(\int_{x\in B_R(0)}+\int_{x\in B^c_R(0)}\right)\frac{A_\infty(A_\infty-A(x))|u_n(x)|^p|u_n(y)|^p}{|x-y|^\mu}dxdy\\
~~~~&+\int\left(\int_{y\in B_R(0)}+\int_{y\in B^c_R(0)}\right)\frac{A(x)(A_\infty-A(y))|u_n(x)|^p|u_n(y)|^p}{|x-y|^\mu}dxdy\\
\to&0,\quad \text{as}~R\to\infty,~n\to\infty.
\endaligned
\end{equation}
Therefore, as $n\to\infty$, it follows from (\ref{1111}) and (\ref{2222}) that
  \begin{equation}\label{c2.26}
I_\infty(u_n) \to m(c),\quad J_\infty(u_n)\to 0.
\end{equation}
where
  \begin{equation*}
 \aligned
 J_\infty(u_n)
=\frac{dI_\infty(u_n^t)}{dt}|_{t=1}
=\|\nabla u_n\|_2^2-\frac{Np-2N+\mu}{2p}\int\int\frac{A_\infty^2|u_n(x)|^p  |u_n(y)|^p}{|x-y|^\mu}dxdy.
\endaligned
\end{equation*}
By Lemma \ref{lm2.13}-(i) and $(\ref{c2.26})$, we have
  \begin{equation}\label{c2.27}
\rho_0^2\leq \|\nabla u_n\|_2^2=\frac{Np-2N+\mu}{2p}\int\int\frac{A_\infty^2|u_n(x)|^p  |u_n(y)|^p}{|x-y|^\mu}dxdy+o(1).
\end{equation}
Using $(\ref{c2.27})$ and Lions' concentration compactness principle \cite{MR778970,MR778974}, we can easily prove that there exist $\delta> 0$ and $\{y_n\}\subset \R^N$
such that $\int_{B_1(y_n)}|u_n|^2dx > \frac{\delta}{2}.$ Let $\hat{u}_n(x) = u_n(x + y_n)$. Then $\|\hat{u}_n\| = \|u_n\|$,
$
\int_{B_1(0)}|\hat{u}_n|^2dx > \frac{\delta}{2},
$
 and
  \begin{equation}\label{2.11.3}
I_\infty(\hat{u}_n)\to m(c),~~J_\infty(\hat{u}_n)\to0.
 \end{equation}
Therefore, there exists $\hat{u}\in H^1(\R^N,\R)\setminus\{0\}$ such that, passing to a subsequence, as $n\to\infty$,
 \begin{equation}\label{2.11.4}
\hat{u}_n\rightharpoonup \hat{u}~~\text{in}~H^1(\R^N,\R),\quad
\hat{u}_n\to \hat{u}~~\text{in}~L^q_{\text{loc}}(\R^N,\R)~\text{for}~ q\in[1,2^*),\quad
\hat{u}_n\to \hat{u}~~\text{a.e.~on}~\R^N.
\end{equation}
Let $w_n = \hat{u}_n -\hat{u}$. Then $(\ref{2.11.4})$ and Lemma \ref{lm2.11} yield
 \begin{equation}\label{2.11.5}
I_\infty(\hat{u}_n)=I_\infty(\hat{u})+I_\infty(w_n)+o(1),~~J_\infty(\hat{u}_n)=J_\infty(\hat{u})+J_\infty(w_n)+o(1).
 \end{equation}
 Set $$\Psi_{\infty}(u):=I_\infty(u)-\frac{1}{2}J_\infty(u)=\frac{Np-2N+\mu-2}{4p}\int\int\frac{A_\infty^2|u(x)|^p|u(y)|^p}{|x-y|^\mu}dxdy,\quad \forall u\in H^1(\R^N,\R).$$
Then by $(\ref{2.10.ii})$, $(\ref{2.11.3})$ and $(\ref{2.11.5})$, we have
 \begin{equation}\label{2.11.6}
  \aligned
m(c)-\Psi_{\infty}(\hat{u}) +o(1)=\Psi_{\infty}(w_n)
\endaligned
 \end{equation}
and
 \begin{equation}\label{2.11.7}
J_\infty(w_n)=-J_\infty(\hat{u})+o(1).
 \end{equation}

By a standard argument \cite[Lemma 2.15]{MR4081327}, we have
\begin{equation}\label{2.11.8}
I_\infty(\hat{u})=m(c),~~J_\infty(\hat{u})=0.
\end{equation}

By Lemma \ref{lm2.8}, there exists $\tilde{t}> 0$ such that $\hat{u}^{\tilde{t}}\in\mathcal{M}(c)$, moreover, it
follows from $(A_1)$, $(\ref{2.11.8})$, and Lemma \ref{cr2.7} that
$$
m(c) \leq I\left(\hat{u}^{\tilde{t}}\right) \leq I_{\infty}\left( \hat{u}^{\tilde{t}}\right) \leq I_{\infty}(\hat{u})=m(c).
$$
This shows that $m(c)$ is achieved at $\hat{u}^{\tilde{t}}\in\mathcal{M}(c)$.

 Case (ii) $\bar{u}\neq 0$. Let $v_n = u_n - \bar{u}$. Then $v_n\rightharpoonup0$ in $H^1(\R^N,\R)$. By Lemma \ref{lm2.11}, we have
\begin{equation}\label{22.11.5}
I(u_n)=I(\bar{u})+I(v_n)+o(1),~~J(u_n)=J(\bar{u})+J(v_n)+o(1).
 \end{equation}
For $u\in H^1(\R^N,\R)$, set
  \begin{equation}\label{A3}
  \begin{array}{ll}
  \Psi(u):&\displaystyle=I(u)-\frac{1}{2}J(u)\vspace{0.2cm}\\
 &\displaystyle=\frac{1}{4p}\int\int\frac{\left[(Np-2N+\mu-2)A(x)-2\nabla A(x)\cdot x\right]A(y)|u(x)|^p  |u(y)|^p}{|x-y|^\mu}dxdy.
 \end{array}
 \end{equation}
 Then it follows from $(A_1)$ and Lemma 3.1 that $\Psi(u) > 0$ for $u \in H^1(\R^N,\R)\setminus\{0\}$.
Using the same argument in \cite[Lemma 2.15]{MR4081327}, we have
 \begin{equation}\label{22.11.8}
I(\bar{u})=m(c),~~J(\bar{u})=0,~~\|\bar{u}\|_2^2=c.
 \end{equation}
This completes the conclusion.
\end{proof}

\begin{lemma}\label{lm2.16} If $\bar{u}\in\mathcal{M}(c)$ and $I(\bar{u}) = m(c)$, then $\bar{u}$ is a critical point of $I|_{S(c)}$.
\end{lemma}
\begin{proof}	
By a similar deformation argument in \cite[Lemma 2.16]{MR4081327}, we get the conclusion.
\end{proof}

{\it Proof of Theorem \ref{th1.2}.} For any $c > 0$, in view of Lemmas \ref{lm2.15} and \ref{lm2.16}, there exists
$\overline{u}_{c} \in \mathcal{M}(c)$ such that
$
I\left(\overline{u}_{c}\right)=m(c),\left.\quad I\right|_{\mathcal{S}(c)} ^{\prime}\left(\overline{u}_{c}\right)=0.
$
In view of the Lagrange multiplier theorem, there exists $\lambda_c\in\R$ such that
$
I^{\prime}\left(\overline{u}_{c}\right)=\lambda_{c} \overline{u}_{c}.
$
Therefore, $(\overline{u}_{c},\lambda_c)$ is a solution of (\ref{1.1.0}). \qed

\vskip4mm
\bibliographystyle{siam}
\bibliography{ChenYYreference(2)}

\begin{thebibliography}{10}

\bibitem{MR2826402}
{\sc J.~Bellazzini and G.~Siciliano}, {\em Scaling properties of functionals
  and existence of constrained minimizers}, J. Funct. Anal., 261 (2011),
  pp.~2486--2507.

\bibitem{bellazzini}
{\sc J.~Bellazzini and N.~Visciglia}, {\em On the orbital stability for a class
  of nonautonomous nls}, Indiana University mathematics journal,  (2010),
  pp.~1211--1230.

\bibitem{bogachev}
{\sc V.~I. Bogachev and M.~A.~S. Ruas}, {\em Measure theory}, vol.~1, Springer,
  2007.

\bibitem{cazenave}
{\sc T.~Cazenave}, {\em Semilinear Schrodinger Equations}, vol.~10, American
  Mathematical Soc., 2003.

\bibitem{MR677997}
{\sc T.~Cazenave and P.-L. Lions}, {\em Orbital stability of standing waves for
  some nonlinear {S}chr\"{o}dinger equations}, Comm. Math. Phys., 85 (1982),
  pp.~549--561.

\bibitem{chen}
{\sc J.~Chen and B.~Guo}, {\em Strong instability of standing waves for a
  nonlocal schr{\"o}dinger equation}, Physica D: Nonlinear Phenomena, 227
  (2007), pp.~142--148.

\bibitem{MR4081327}
{\sc S.~Chen and X.~Tang}, {\em Normalized solutions for nonautonomous
  {S}chr\"{o}dinger equations on a suitable manifold}, J. Geom. Anal., 30
  (2020), pp.~1637--1660.

\bibitem{feng}
{\sc B.~Feng and X.~Yuan}, {\em On the cauchy problem for the
  schr{\"o}dinger-hartree equation}, Evolution Equations \& Control Theory, 4
  (2015), p.~431.

\bibitem{MR1430506}
{\sc L.~Jeanjean}, {\em Existence of solutions with prescribed norm for
  semilinear elliptic equations}, Nonlinear Anal., 28 (1997), pp.~1633--1659.

\bibitem{MR2561169}
{\sc E.~Lenzmann}, {\em Uniqueness of ground states for pseudorelativistic
  {H}artree equations}, Anal. PDE, 2 (2009), pp.~1--27.

\bibitem{MR3390522}
{\sc G.-B. Li and H.-Y. Ye}, {\em The existence of positive solutions with
  prescribed {$L^2$}-norm for nonlinear {C}hoquard equations}, J. Math. Phys.,
  55 (2014), pp.~121501, 19.

\bibitem{MR471785}
{\sc E.~H. Lieb}, {\em Existence and uniqueness of the minimizing solution of
  {C}hoquard's nonlinear equation}, Studies in Appl. Math., 57 (1976/77),
  pp.~93--105.

\bibitem{MR1817225}
{\sc E.~H. Lieb and M.~Loss}, {\em Analysis}, vol.~14 of Graduate Studies in
  Mathematics, American Mathematical Society, Providence, RI, second~ed., 2001.

\bibitem{MR778970}
{\sc P.-L. Lions}, {\em The concentration-compactness principle in the calculus
  of variations. {T}he locally compact case. {I}}, Ann. Inst. H. Poincar\'{e}
  Anal. Non Lin\'{e}aire, 1 (1984), pp.~109--145.

\bibitem{MR778974}
\leavevmode\vrule height 2pt depth -1.6pt width 23pt, {\em The
  concentration-compactness principle in the calculus of variations. {T}he
  locally compact case. {II}}, Ann. Inst. H. Poincar\'{e} Anal. Non
  Lin\'{e}aire, 1 (1984), pp.~223--283.

\bibitem{MR1649671}
{\sc I.~M. Moroz, R.~Penrose, and P.~Tod}, {\em Spherically-symmetric solutions
  of the {S}chr\"{o}dinger-{N}ewton equations}, vol.~15, 1998, pp.~2733--2742.
\newblock Topology of the Universe Conference (Cleveland, OH, 1997).

\bibitem{MR3056699}
{\sc V.~Moroz and J.~Van~Schaftingen}, {\em Groundstates of nonlinear
  {C}hoquard equations: existence, qualitative properties and decay
  asymptotics}, J. Funct. Anal., 265 (2013), pp.~153--184.

\bibitem{pekar}
{\sc S.~I. Pekar}, {\em Untersuchungen {\"u}ber die Elektronentheorie der
  Kristalle}, Akademie-verlag, 1954.

\bibitem{MR3642765}
{\sc H.~Ye}, {\em Mass minimizers and concentration for nonlinear {C}hoquard
  equations in {$\Bbb R^N$}}, Topol. Methods Nonlinear Anal., 48 (2016),
  pp.~393--417.

\end{thebibliography}
\end{document}